\documentclass[12pt]{amsart}
\usepackage[margin=3cm]{geometry}

\usepackage[utf8]{inputenc}
\usepackage[english]{babel}
\usepackage{amsmath,amssymb,amsthm}
\usepackage{enumerate,mathrsfs}
\usepackage{hyperref}
\usepackage{listings}

\newtheorem{theorem}{Theorem}[section]

\newtheorem{lemma}[theorem]{Lemma}
\newtheorem{proposition}[theorem]{Proposition}
\newtheorem{definition}[theorem]{Definition}

\theoremstyle{definition}

\newtheorem{remark}[theorem]{Remark}

\DeclareMathOperator{\Aut}{Aut}

\DeclareMathOperator{\id}{id}

\newcommand{\PG}{\mathrm{PG}}

\newcommand{\K}{\mathbb{K}}

\usepackage{graphicx}
\usepackage{tikz}

\title[Small light dual multinets]{Light dual multinets of order six in the projective plane}
\author{Norbert Bogya}
\address{Bolyai Institute \\
University of Szeged \\
Aradi v\'ertan\'uk tere 1\\
H-6720 Szeged, Hungary}
\email{nbogya@math.u-szeged.hu}
\author{G\'abor P. Nagy}
\address{Department of Algebra \\
Budapest University of Technology and Economics\\
Egry J\'ozsef utca 1\\
H-1111 Budapest, Hungary}
\address{Bolyai Institute \\
University of Szeged \\
Aradi v\'ertan\'uk tere 1\\
H-6720 Szeged, Hungary}
\email{nagyg@math.bme.hu}
\thanks{Research supported by NKFIH-OTKA Grants 114614 and 119687.}
\date{\today}

\begin{document}

\begin{abstract}
The aim of this paper is twofold: First we classify all abstract light dual multinets of order $6$ which have a unique line of length at least two. Then we classify the weak projective embeddings of these objects in projective planes over fields of characteristic zero. For the latter we present a computational algebraic method for the study of weak projective embeddings of finite point-line incidence structures.
\end{abstract}

\maketitle

\section{Introduction}

In recent years, nets realizing a finite group have been investigated in connection with complex line arrangements and resonance theory; see \cite{fy2007,knp_3,knp_k,miq,per,ys2004,ys2007}. The concept of a multinet was introduced by Falk and Yuzvinsky \cite{fy2007} as multi-arrangements of lines in the complex projective plane, with a partition into three or more equinumerous classes which have equal multiplicities at each inter-class intersection point, and satisfy a connectivity condition. Korchmáros and Nagy \cite{KN_multinets} gave a more formal definition for a dual multinet of the projective plane, labeled by a quasigroup $Q$. Let $\K$ be a field, $Q$ a quasigroup and for $i=1,2,3$, let $\alpha_i:Q\to \PG(2,\K)$ be maps such that the points $\alpha_1(x)$, $\alpha_2(y)$ and $\alpha_3(x\cdot y)$ are collinear for all $x,y\in Q$. Define the \textit{multisets} $\Lambda_i=\alpha_i(Q)$, $i=1,2,3$. Then $(\Lambda_1,\Lambda_2,\Lambda_3)$ is a \textit{dual multinet,} labeled by $Q$. If the maps $\alpha_i$ are injective and their images $\Lambda_i$ are disjoint, then the dual multinet is called \textit{light.} This terminology is in accordance with the one introduced by Bartz and Yuzvinsky \cite{Bartz2013,BartzYuz2014}. 

Let $(\Lambda_1,\Lambda_2,\Lambda_3)$ be a light dual multinet in $\PG(2,\K)$ labeled by the quasigroup $Q$. As we show later in the abstract setting, if the line $\ell$ intersects two components $\Lambda_i,\Lambda_j$ then there is a integer $r$ such that $r=|\ell\cap \Lambda_1|=|\ell\cap \Lambda_2|=|\ell\cap \Lambda_3|$; this integer $r$ is called the \textit{length} of $\ell$ w.r.t. $(\Lambda_1,\Lambda_2,\Lambda_3)$. Examples of group-labeled light dual multinets were given by Bartz and Yuzvinsky \cite{Bartz2013,BartzYuz2014} and by Korchm\'aros and Nagy \cite{KN_multinets}. In \cite{BartzYuz2014}, the examples of light dual multinets are constructed from a three dimensional multinet. They are labeled by the dihedral group of order $n$ and the long lines have order $2$ or $n/2$. In \cite{KN_multinets} the authors define triangular light dual multinets of order $n$ which are contained in the union of three lines of length $n/3$, and tetrahedron and conic-line type light dual multinets, which have a unique line of length $n/2$. 

The motivation of the present paper is the following observation. The known constructions of light dual multinets and the results of \cite{KN_multinets} suggest that the length $r>1$ of lines of the light dual multinet makes a big difference in their geometric structure. While for $r\geq 9$, the light dual multinet is well structured in geometric and algebraic sense, the case of small $r$, especially $r=2$ shows many irregularities. In this paper, we classify all abstract light dual multinets of order $6$ with a unique line of length $r>1$. Moreover, we compute all possible realizations of these abstract light dual multinets in projective planes over fields of characteristic $0$. For this purpose, we present an algebraic framework to handle the occurring system of polynomial equations. The computation was done using computer algebra systems GAP \cite{GAP4}, Singular \cite{Singular} and SageMath \cite{Sage8.2}. 

The paper is organized as follows. In section 2, we give the basic definitions and terminology and present some simple results which are known and more or less folklore. In Section 3, we introduce the concept of an abstract light dual multinet $\Sigma$, labeled by a quasigroup $Q$ and explain the relation between long lines of $\Sigma$ and subsquares of $Q$. In Section 4, we classify all abstract light dual multinets of order $6$ having a unique line of length $r>1$. Section 5 presents the algebraic machinery for the study of weak embeddings of point-line incidence structures in the projective plane $\PG(2,\K)$, where $\K$ is a field. In Sections 6 and 7 we apply this machinery to classify the weak projective embeddings of abstract light dual multinets of order $6$. In the appendices we give the SageMath \cite{Sage8.2} codes which compute the results.  

%

\section{Preliminaries}

\subsection{Quasigroups, isotopes, parastrophes}
A \textit{quasigroup} $(Q,\cdot)$ is a set endowed with a binary operation $x\cdot y$ such that the equation $x\cdot y=z$ can be uniquely solved if any two of the three values $x,y,z\in Q$ are given. One denotes the solutions with left and right division: $y=x\setminus z$, $x=z/y$. The maps $L_a:x\mapsto a\cdot x$, $R_a:x\mapsto x\cdot a$ are the \textit{left} and \textit{right multiplication maps} of $Q$. 

Let $(Q,\cdot)$ and $(R,\circ)$ be quasigroups. An \textit{isostrophism} from $Q$ to $R$ is a quadruple $(\sigma,\gamma_1,\gamma_2,\gamma_3)$, where $\gamma_1,\gamma_2,\gamma_3$ are bijective maps from $Q$ to $R$ and $\sigma \in S_3$ such that for all elements $x_1,x_2,x_3 \in Q$ the following holds: 
\[x_{\sigma(1)}\cdot x_{\sigma(2)}=x_{\sigma(3)} \Longleftrightarrow \gamma_1(x_1)\circ\gamma_2(x_2) = \gamma_3(x_2).\]
It is straightforward to see that being isostrophic is an equivalence
relation. If $\sigma$ is the identity then we speak of the \textit{isotopism} $(\gamma_1,\gamma_2,\gamma_3)$; for all $x,y \in Q$ 
\[\gamma_1(x)\circ\gamma_2(y)=\gamma_3(x\cdot y).\]
If the underlying sets $Q, R$ are equal and $\gamma_1,\gamma_2,\gamma_3$ are the identical maps than we say that the quasigroups $(Q,\cdot)$ and $(Q,\circ)$ are \textit{$\sigma$-conjugate}. We speak of a \textit{principal isotopism} if the underlying sets $Q,R$ are equal, $\sigma=\id$, $\alpha_1=R_v$, $\alpha_2=L_u$ and $\alpha_3=\id$ for some fixed elements $u,v\in Q$. Then for all $x,y\in Q$
\[x\circ y= x/u \cdot v\backslash y.\]
In this case, $(Q,\circ)$ is a loop with unit element $v\cdot u$. 

We remark that some authors use the term \textit{parastrophe} for isostrophe. In the language of Latin squares, one says that isostrophic quasigroups $Q$ and $R$ belong to the same \textit{main class.} 

\subsection{Subsquares in quasigroups}
While subquasigroups are defined in the obvious way, quasigroups have another important substructure which we call \textit{subsquares.} In \cite{DenesKeedwell},  subsquares of order $2$ are called \textit{intercalates. }

\begin{definition}
Let $(Q,\cdot)$ be a quasigroup and $S_1,S_2,S_3\subseteq Q$ such that
\[S_1\cdot S_2\subseteq S_3, \quad S_1\setminus S_3\subseteq S_2, \quad S_3/S_2\subseteq S_1.\] 
Then we say that the triple $(S_1,S_2,S_3)$ is a \textit{subsquare} of $Q$. For a subsquare one has $|S_1|=|S_2|=|S_3|$; this cardinality is the \textit{order} of the subsquare. Subsquares of order $1$, and the subsquare $(Q,Q,Q)$ are called trivial. Nontrivial subsquares are proper. 
\end{definition}

Clearly, subsquares of $Q$ form a poset and the intersection of subsquares is a subsquare, as well. Moreover, for $U_1,U_2,U_3\subseteq Q$, the subsquare generated by $(U_1,U_2,U_3)$ is the smallest subsquare $(S_1,S_2,S_3)$ such that $U_i\subseteq S_i$, $i=1,2,3$. For any fixed subsquare $(S_1,S_2,S_3)$ of $Q$, there is a principal loop isotope $(Q,\oplus)$ of $(Q,\cdot)$ in which $S_3$ is a subloop. Indeed, take arbitrary elements $u\in S_1$, $v\in S_2$ and put
\[x\oplus y=x/v \cdot u\backslash y.\]
Then $(Q,\oplus)$ is a loop with unit element $u\cdot v$, and
\[S_3\oplus S_3 = S_3/v \cdot u\backslash S_3 = S_1 \cdot S_2 =S_3.\]

For a general loop of order $n$, the order of a subloop does not divide $n$. However, proper subloops have order at most $n/2$ and subloops of order $n/2$ are normal.

\subsection{Point-line incidence structures}

We quote the basic definitions on (point-line) incidence structures from \cite[Chapter I]{DesignTheory}. A \textit{point-line incidence structure} is a triple $(\mathcal{P},\mathcal{B},I)$, where $\mathcal{P},\mathcal{B}$ are sets and $I\subseteq \mathcal{P} \times \mathcal{B}$. The elements of $\mathcal{P}$ are called \textit{points}, the elements of $\mathcal{B}$ are \textit{blocks} or \textit{lines.} Instead of $(p, B) \in I$ we will simply write $pIB$ and use such geometric language as ``the point $p$ lies on the block $B$'', ``$B$ passes through $p$'', ``$p$ and $B$ are incident'', etc. The \textit{trace} of the block $B$ is the set $T(B)$ of points, incident with $B$. An incidence structure is called \textit{simple,} if $T(B)=T(C)$ implies $B=C$ for all blocks $B,C$. For a simple incidence structure, we can identify each block $B$ with the corresponding point set $T(B)$ and the incidence relation $I$ with the membership relation $\in$. In this case, we write $(\mathcal{P},\mathcal{B})$ for $(\mathcal{P},\mathcal{B},\in)$. 

Let $(\mathcal{P},\mathcal{B},I)$ be an incidence structure. We say that the points $p_1,p_2,\ldots,p_k\in \mathcal{P}$ are \textit{collinear} if there is a block $B\in\mathcal{B}$ such that $p_i I B$ for all $i=1,\ldots,k$. By a \textit{collinear triple} we mean three different collinear points. It is easy to see that in a simple incidence structure, all blocks of size at least three can be reconstructed from the set of collinear triples. 

An isomorphism between the incidence structures $(\mathcal{P}_1,\mathcal{B}_1,I_1)$ and $(\mathcal{P}_2,\mathcal{B}_2,I_2)$ is a bijective map $\varphi: \mathcal{P}_1\cup \mathcal{B}_1 \to \mathcal{P}_2\cup \mathcal{B}_2$ preserving points, lines and the incidence relation. If both incidence structures are simple and blocks are identified with subsets of points, then an isomorphism is a bijective map $\mathcal{P}_1 \to \mathcal{P}_2$, which induces a bijection from $\mathcal{B}_1$ to $\mathcal{B}_2$. If the incidence structures are simple and all blocks have size at least $3$, then an isomorphism is a bijective map $\mathcal{P}_1 \to \mathcal{P}_2$, which induces a bijection on the sets of collinear triples. Finally, we mention that the \textit{dual} of the incidence structure $(\mathcal{P},\mathcal{B},I)$ is the incidence structure $(\mathcal{B},\mathcal{P},I')$, where $pIB$ holds if and only if $BI'p$.

\subsection{Abstract dual 3-nets, well-indexing}
An abstract dual $3$-net is a simple point-line incidence structure $\Sigma=(\mathcal{P},\mathcal{L})$, where $\mathcal{P}$ is the disjoint union of the three subsets $\mathcal{P}_1$, $\mathcal{P}_2$ and $\mathcal{P}_3$ and $\mathcal{L}$ consists of subsets of $\mathcal{P}$ of size $3$ such that for any $p_i\in \mathcal{P}_i$, $p_j\in \mathcal{P}_j$ ($1\leq i<j\leq 3$) there is exactly one element of $\mathcal{L}$, which contains $p_i$ and $p_j$. Given a quasigroup $(Q,\cdot)$, one constructs an abstract dual $3$-net in the following way. One takes three disjoint sets  $\mathcal{P}_1,\mathcal{P}_2, \mathcal{P}_3$ with cardinality $|Q|$, bijections $\alpha_i:Q\to \mathcal{P}_i$, and defines $\mathcal{P}=\mathcal{P}_1\cup \mathcal{P}_2 \cup \mathcal{P}_3$, 
\[\mathcal{L}=\{\{\alpha_1(x), \alpha_2(y), \alpha_3(x\cdot y)\} \mid x,y\in Q\}.\]
Moreover, any abstract dual $3$-net $\Sigma$ can be obtained in this way. We say that $\Sigma$ is \textit{labeled} by the quasigroup $Q$. 

The labeling quasigroup of the abstract dual $3$-net $\Sigma$ is not uniquely determined. Assume that $\Sigma$ is labeled by $(Q,\cdot)$ w.r.t. the labeling maps $(\alpha_1,\alpha_2,\alpha_3)$. Let $(\sigma,\gamma_1,\gamma_2,\gamma_3)$ be an isostrophism from $Q$ to $(R,\circ)$. Then we define the maps $\hat{\alpha}_i: R\to \mathcal{P}$, $i=1,2,3$, by $\hat{\alpha}_{\sigma(i)} = \alpha_i \gamma_{\sigma(i)}^{-1}$. The triple $(\hat{\alpha_1},\hat{\alpha_2},\hat{\alpha_3})$ turns out to be a labeling of $\Sigma$ by $R$. Conversely, if $\Sigma$ can be labeled by the quasigroups $(Q,\cdot)$ and $(R,\oplus)$ via the labeling maps $\alpha_1,\alpha_2,\alpha_3$ and $\hat{\alpha}_1, \hat{\alpha}_2, \hat{\alpha}_3$, then $Q$ and $R$ are isostrophes. If $Q$ is a group (associative multiplication) and $Q$ and $R$ are isostrophes, then $R$ is a group isomorphic to $Q$. 

We say that the abstract dual $3$-net $\Sigma=(\mathcal{P},\mathcal{L})$ of order $n$ is \textit{well-indexed,} if the following hold:
\begin{enumerate}[(W1)]
\item $\mathcal{P}$ is the disjoint union of $\mathcal{P}_1=\{1,\ldots,n\}$, $\mathcal{P}_2=\{n+1,\ldots,2n\}$, and $\mathcal{P}_3=\{2n+1,\ldots,3n\}$.
\item For each $j\in \{1,\ldots,n\}$, the triples $\{1,n+j,2n+j\}$, $\{j,n+1,2n+j\}$ are in $\mathcal{L}$. 
\end{enumerate}
For any finite $\Sigma$, (W1) can be assumed without loss of generality. Moreover, (W2) can be achieved by rearranging first $\mathcal{P}_3$ and then $\mathcal{P}_1$. Also notice that a well-indexed dual $3$-net has a canonical labeling by a loop $(L,*)$, where the underlying set of $L$ is $\{1,\ldots,n\}$ and the unit element is $1$. Indeed, we take the labeling maps
\[\alpha_1(x)=x, \quad \alpha_2(y)=n+y, \quad \alpha_3(z)=2n+z.\]
\begin{lemma} \label{lm:wellindexing1}
Let $\Sigma=(\mathcal{P},\mathcal{L})$ be an arbitrary finite abstract dual $3$-net, labeled by the quasigroup $(Q,\cdot)$. Then there is an isomorphic well-indexed finite dual $3$-net $\Sigma^*=(\mathcal{P}^*,\mathcal{L}^*)$. 
\end{lemma}
\begin{proof}
Let us first relabel $\Sigma$ with a loop isotope $(Q,\circ,e)$ of $(Q,\cdot)$. Write $Q=\{q_1=e,q_2,\ldots,q_n\}$ and define the binary operation $i*j$ by $q_i\circ q_j=q_{i*j}$. Then, $(\{1,\ldots,n\},*)$ is a loop, isomorphic to $(Q,\circ)$ and we can relabel $\Sigma$ 
with $(\{1,\ldots,n\},*)$. Define the abstract dual $3$-net $\Sigma'$ of $(\{1,\ldots,n\},*)$ with labeling maps 
\[\alpha_1(x)=x, \quad \alpha_2(y)=n+y, \quad \alpha_3(z)=2n+z.\]
Clearly, $\Sigma'$ is well-indexed and $\Sigma$ and $\Sigma'$ are isomorphic. 
\end{proof}

\subsection{Projective realizations of finite quasigroups} In the form we are interested in, projective realizations of quasigroups have been introduced by Yuzvisnky \cite{ys2004}. However, the idea goes back to the study of the additive and multiplicative loops of ternary rings and their geometrical interpretation in non Desarguesian projective planes, see \cite{BarlStram,DesignTheory,DenesKeedwell}. Let $\K$ be an algebraically closed field and $Q$ a finite quasigroup of order $n$. We say that the disjoint point sets $\Lambda_1,\Lambda_2,\Lambda_3$ of $\PG(2,\K)$ \textit{realize} $Q$ if there are bijections $\alpha_i:Q\to \Lambda_i$, $i=1,2,3$, such that for $x,y,z\in Q$, the points $\alpha_1(x)$, $\alpha_2(y)$, $\alpha_3(z)$ are collinear if and only if $x\cdot y=z$ holds. It is clear that $(\Lambda_1,\Lambda_2,\Lambda_3)$ can be seen as an abstract dual $3$-net, embedded in $\PG(2,\K)$ and labeled by $Q$. Important examples of quasigroup realizations were given by Yuzvinsky \cite{ys2004} when $Q$ is an abelian group, by Stipins \cite{sj2004} when $Q$ is a nonassociative loop of order $5$, by Korchm\'aros, Nagy and Pace when $Q$ is a finite dihedral group, and by Urz\'ua \cite{urzua2009} when $Q$ is the quaternion group of order $8$. In fact, it turns out that these are essentially all projective realizations of finite groups.

Let $\Lambda=(\Lambda_1,\Lambda_2,\Lambda_3)$ be a projective realization of the finite group $G$. $\Lambda$ is called \textit{algebraic,} if $\Lambda_1\cup \Lambda_2\cup \Lambda_3$ is contained in a cubic curve $\mathcal{C}$. If $\mathcal{C}$ is the union of a line and an irreducible conic then $\Lambda$ is of \textit{conic-line type}. If $\mathcal{C}$ is the union of three lines $\ell_1,\ell_2,\ell_3$ then $\Lambda$ is of \textit{triangular} or of \textit{pencil} type, depending on if $\ell_1,\ell_2,\ell_3$ form a triangle or have a point in common. Finally, we say that $\Lambda$ is of \textit{tetrahedron} type, if $\Lambda_1\cup \Lambda_2\cup \Lambda_3$ is contained in the union of the six lines of a complete quadrilateral. In this case, each component $\Lambda_i$ is contained in the union of two lines. 

The following almost complete classification of such $3$-nets is proven in \cite{knp_3}.
\begin{theorem}
\label{mainteo} In the projective plane $PG(2,\mathbb{K})$ defined over an algebraically closed field $\mathbb{K}$ of characteristic $p\geq 0$, let $(\Lambda_1,\Lambda_2,\Lambda_3)$ be a dual $3$-net of order $n\geq 4$ which realizes a group $G$. If either $p=0$ or $p>n$ then one of the following holds.
\begin{enumerate}[(I)]
\item $G$ is either cyclic or the direct product of two cyclic groups, and $(\Lambda_1, \Lambda_2, \Lambda_3)$ is algebraic.
\item $G$ is dihedral and $(\Lambda_1,\Lambda_2,\Lambda_3)$ is of tetrahedron type.
\item $G$ is the quaternion group of order $8$.
\item $G$ has order $12$ and is isomorphic to $\rm{Alt}_4$.
\item $G$ has order $24$ and is isomorphic to $\rm{Sym}_4$.
\item $G$ has order $60$ and is isomorphic to $\rm{Alt}_5$.
\end{enumerate}
\end{theorem}
A computer aided exhaustive search shows that if $p=0$ then (IV) (and hence (V), (VI)) does not occur; see \cite{np}. It has been conjectured that this holds true in any characteristic. 

Not much is known about projective realizations of nonassociative quasigroups. Nonassociative quasigroups of order $5$ are all isostrophes, their projective realization was given by Stipins \cite{sj2004}. Up to isostrophy there are $12$ quasigroups of order $6$, Urz\'ua \cite{urzua2009} computed their realizations for $\K=\mathbb{C}$, or showed that no such realization exists.

\section{Abstract light dual multinets}

We define a generalization of the concept of an abstract dual $3$-net. 
\begin{definition}
An abstract light dual multinet labeled by the quasigroup $Q$ is a pair $(\mathcal{P},\mathcal{M})$ with the following properties:
\begin{enumerate}[(1)]
\item $\mathcal{P}$ is the disjoint union of the three subsets $\mathcal{P}_1$, $\mathcal{P}_2$ and $\mathcal{P}_3$.
\item $\mathcal{M}$ consists of subsets of $\mathcal{P}$ such that for any ${p},{q}\in \mathcal{P}$ there is at most one element of $\mathcal{M}$, which contains ${p}$ and ${q}$. 
\item There are bijections $\alpha_i:Q\to \mathcal{P}_i$ ($i=1,2,3$), such that for any $x,y\in Q$ there is an element $S\in \mathcal{M}$ containing $\alpha_1(x)$, $\alpha_2(y)$ and $\alpha_3(x\cdot y)$. 
\item $|\mathcal{M}|>1$. 
\end{enumerate}
\end{definition}

In the sequel, $(\mathcal{P},\mathcal{M})$ denotes an abstract light dual multinet, labeled by the finite quasigroup $Q$ of order $n$. The labeling maps are $\alpha_1,\alpha_2,\alpha_3$. We usually denote the multinet just by $\mathcal{M}$. We call the elements of $\mathcal{P}$ \textit{points,} and the elements of $\mathcal{M}$ \textit{lines} or \textit{blocks}. Property (4) prohibits a light dual multinet to be degenerate, that is, no block can contain all points.

\begin{lemma} \label{lm:lengthdef}
Let $B\in\mathcal{M}$ be a block of the abstract light dual multinet $(\mathcal{P},\mathcal{M})$. Let $S_i=\{x\in Q\mid \alpha_i(x)\in B\}$ be the subsets of $Q$ corresponding to the intersections of $B$ and $\mathcal{P}_i$ ($i=1,2,3$). Then, $(S_1,S_2,S_3)$ is a subsquare of $Q$; its order is called the length of $B$ with respect to $(\mathcal{P},\mathcal{M})$.
\end{lemma}
\begin{proof}
For any $x\in S_1$, $y\in S_2$, we have $\alpha_3(x\cdot y) \in B$, hence $S_1\cdot S_2\subseteq S_3$. In a similar way one shows $S_1\setminus S_3\subseteq S_2$ and $S_3/S_2\subseteq S_1$. 
\end{proof}

This lemma enables us to define a special class of abstract light dual multinets. Let $(Q,\cdot)$ be a quasigroup and $(S_1,S_2,S_3)$ be a subsquare of $Q$. Let $(\mathcal{P},\mathcal{N})$ be the abstract dual $3$-net of $Q$. Define the "superline"
\[B'=\alpha_1(S_1) \cup \alpha_2(S_2) \cup \alpha_3(S_3)\]
as the union of $r^2$ lines of $\mathcal{N}$, where $r=|S_1|=|S_2|=|S_3|$. Put
\[\mathcal{M}= \{B\in \mathcal{N} \mid 1\geq |B\cap \alpha_i(S_i)| \} \cup \{B'\}. \]
Then $\mathcal{M}$ is a $Q$-labeled abstract light dual multinet with $n^2-r^2$ lines of length $1$ and one line of length $r$. We denote this abstract light dual multinet by $\mathcal{M}_{S_1,S_2,S_3}$. In the special case when $S_1=S_2=S_3=S$ is a subquasigroup of $Q$, we write $\mathcal{M}_S$.

The labeling quasigroup of $\mathcal{M}$ is not uniquely determined. In the same way as we showed for dual $3$-nets, one can show that if the light dual multinet $\mathcal{M}$ is labeled by the quasigroup $Q$ and $R$ is an isostrophic quasigroup with $Q$ then $\mathcal{M}$ can be labeled by $R$ as well. Conversely, assume that $\mathcal{M}$ can be labeled by the quasigroup $(Q,\cdot)$ and $(R,\oplus)$.
Then we call the quasigroups $Q$ and $R$ are \textit{$\mathcal{M}$-isostrophes. }

\begin{lemma} \label{lm:wellindexing2}
Let $\Sigma=(\mathcal{P},\mathcal{M})$ be a finite abstract light dual multinet of order $n$, labeled by the quasigroup $(Q,\cdot)$. Let $\ell$ be a superline of length $r$. Then $\Sigma$ is isomorphic to an abstract light dual multinet $\Sigma'=(\mathcal{P}',\mathcal{M}')$ such that the following hold:
\begin{enumerate}[(WM1)]
\item $\mathcal{P}'$ is the disjoint union of $\mathcal{P}'_1=\{1,\ldots,n\}$, $\mathcal{P}'_2=\{n+1,\ldots,2n\}$, and $\mathcal{P}'_3=\{2n+1,\ldots,3n\}$.
\item $\ell'=\{1,\ldots,r, n+1,\ldots,n+r, 2n+1,\ldots,2n+r\}$ is a superline of length $r$ of $\Sigma'$.
\item For each $j\in \{r+1,\ldots,n\}$, the triples $1,n+j,2n+j$ and $j,n+1,2n+j$ are collinear in $\mathcal{M}'$. 
\end{enumerate}
\end{lemma}
\begin{proof}
The proof is similar to the proof of Lemma \ref{lm:wellindexing1}, with two additional considerations. First, the loop isotope $(Q,\circ,e)$ must be chosen such that the superline $\ell$ corresponds to a subloop $S$. Second, the elements $q_1,\ldots,q_n$  of $Q$ must be chosen such that $q_1=e$ and $S=\{q_1,\ldots,q_r\}$. 
\end{proof}

\begin{definition}
An abstract light dual multinet $\Sigma'$ is said to be \textit{well-indexed w.r.t. a superline of length $r$,} if conditions (WM1)-(WM3) of Lemma \ref{lm:wellindexing2} hold. 
\end{definition}

\section{Classification of abstract light dual multinets of order $6$}

We start with an obvious lemma.

\begin{lemma} \label{lm:subsquares}
Let $Q$ be a quasigroup of order $6$. Then all proper subsquares of $Q$ have order $1$, $2$ or $3$. A subsquare of order $3$ cannot contain a subsquare of order $2$. Any proper subsquare can be generated by a triple of the form $(\{x,y\},\{z\},\emptyset)$, where $x,y,z\in Q$, $x\neq y$. \qed
\end{lemma}

Lemma \ref{lm:lengthdef} and Lemma \ref{lm:subsquares} imply that a nontrivial abstract light dual multinet of order $6$ can only have lines of length $1$, $2$ or $3$. Moreover, Lemma \ref{lm:subsquares} provides us an effective method to generate all proper subsquares of a quasigroup of order $6$. The number of subsquares of order $2$ and $3$ of quasigroups of order $6$ are given in \cite[Figure 4.2.2]{DenesKeedwell}.

\begin{table} 
\caption{Isomorphism classes of abstract light dual multinets with one superline\label{table:ldms_classification}}
\begin{tabular}{|c|c|l|c|}
\hline
id&(A)&(B)&$\Aut(\mathcal{M}_{S_1,S_2,S_3})$\\ \hline
$M_{1}$ & 3 & 6.1, 6.9 & $((C_3 \times C_3 \times C_3) : C_3) : (C_2 \times C_2)$ \\ 
$M_{2}$ & 3 & 6.2, 6.3, 6.9  & $((C_3 \times C_3 \times C_3) : C_3) : (C_2 \times C_2)$ \\ 
$M_{3}$ & 2 & 6.1, 6.4  & $C_2 \times S_4$ \\ 
$M_{4}$ & 2 & 6.2, 6.5 & $C_2 \times S_4$ \\ 
$M_{5}$ & 2 & 6.4 & $C_2 \times S_4$ \\ 
$M_{6}$ & 2 & 6.7 & $C_2 \times S_4$ \\ 
$M_{7}$ & 2 & 6.5 & $C_2 \times D_8$ \\ 
$M_{8}$ & 2 & 6.8, 6.11 & $C_2 \times D_8$ \\ 
$M_{9}$ & 2 & 6.10 & $C_2 \times D_8$ \\ 
$M_{10}$ & 2 & 6.11 & $C_2 \times D_8$ \\ 
$M_{11}$ & 2 & 6.9, 6.12 & $C_2 \times C_2 \times C_2$ \\ 
$M_{12}$ & 2 & 6.12 & $C_2 \times C_2 \times C_2$ \\ 
$M_{13}$ & 2 & 6.6, 6.7 & $S_3$ \\ 
$M_{14}$ & 2 & 6.5, 6.10 & $C_2 \times C_2$ \\ 
$M_{15}$ & 2 & 6.7, 6.10 & $C_2 \times C_2$ \\ 
$M_{16}$ & 2 & 6.11, 6.12 & $C_2 \times C_2$ \\ \hline 
\end{tabular}
\end{table}

\begin{theorem} \label{thm:16classes}
There are $16$ isomorphism classes of abstract light dual multinets with a unique superline. In Table \ref{table:ldms_classification} we listed (A) the length of the superlines, (B) the D\'enes-Keedwell numbers of the labeling quasigroups, and the structure of the automorphism group. 
\end{theorem}
\begin{proof}
Using the GAP4 \cite{GAP4} computer algebra system and the GAP package DESIGN \cite{DESIGN}, we can construct all abstract light dual multinets of order $6$ with one superline. The commands can be accessed using the SageMath \cite{Sage8.2} interface. 
All computations are done with well-indexed abstract light dual multinets.
\end{proof}

\begin{proposition} \label{pr:superlinr3}
Let $(\mathcal{P},\mathcal{M})$ be an abstract light dual multinet of order $6$, which contains a superline of length $3$. Then $\mathcal{M}$ can be labeled either by the cyclic group or by the dihedral group of order $6$. 
\end{proposition}
\begin{proof}
The claim follows from Theorem \ref{thm:16classes}. 
\end{proof}

\section{Weak projective embeddings of incidence structures}

In this section, $\Sigma=(\mathcal{P},\mathcal{L})$ denotes a simple incidence structure and $\K$ a field. Moreover, we assume that all lines of $\Sigma$ are incident with at least $3$ points. 

\subsection{Weak projective embeddings} 
\begin{definition}
A map $\beta:\mathcal{P} \to \PG(2,\mathbb{L})$ is a \textit{weak projective $\K$-embedding} of $\Sigma$ if the following hold:
\begin{enumerate}[(1)]
\item $\mathbb{L}/\K$ is a field extension.
\item $\beta$ is injective and $\beta(\mathcal{P})$ spans $\PG(2,\mathbb{L})$.
\item For any collinear triple $p_1,p_2,p_3\in \mathcal{P}$, the points $\beta(p_1)$, $\beta(p_2)$, $\beta(p_3)$ are collinear. 
\end{enumerate}
The extension field $\mathbb{L}$ is called the \textit{coordinate field} of $\beta$. 
\end{definition}

The definition implies that for a weak projective $\K$-embedding $\beta$, noncollinear points of $\Sigma$ can have collinear images in the projective plane. This is formalized in the following definition. 

\begin{definition}
Let $\beta:\mathcal{P}\to \PG(2,\mathbb{L})$ be a weak projective $\K$-embedding of the incidence structure $\Sigma=(\mathcal{P},\mathcal{L})$. We say that $S\subseteq \mathcal{P}$ is a \textit{merged block of $\beta$,} if $|S|>1$, $S\not\in \mathcal{L}$ and 
\[S=\beta^{-1}(\ell)\]
for some line $\ell$ of $\PG(2,\mathbb{L})$. 
\end{definition}
We remark that $\beta$ is a projective realization of a finite quasigroup if and only if each merged block is contained in one of the component $\beta(\mathcal{P}_i)$, $i=1,2,3$. For example, if $\beta$ is a projective realization of a dihedral group of order $2m$, then the component $\beta(\mathcal{P}_i)$ ($i=1,2,3$) is contained in the union of two lines, cf. \cite[Proposition 22]{knp_3}. In other words, the projective realizations of finite dihedral groups are of \textit{tetrahedron type.}

\subsection{Pre-embeddings and ring homomorphisms}
Let $t_1,\ldots,t_n$ be indeterminates over $\K$, $R=\K[t_1,\ldots,t_n]$. Let 
\[\xi:\mathcal{P}\to \PG(2,\K(t_1,\ldots,t_n))\]
be an injective map. Then for each ${p}\in \mathcal{P}$, there are  polynomials $u_{p},v_{p},w_{p}\in R$ such that $\xi({p})$ can be given by the vector $\mathbf{u}_{p}=[u_{p},v_{p},w_{p}]$. We can assume that $\gcd(u_{p},v_{p},w_{p})=1$ for all ${p}$. 

Let $\mathbb{L}$ be a field extension of $\K$ and $\tau_1,\ldots,\tau_n\in \mathbb{L}$. The substitution $t_i=\tau_i$ determines a map 
\[\xi|_{(t_i)=(\tau_i)}:\mathcal{P}\to \PG(2,\mathbb{L}).\]
This map may be not well-defined. Clearly, any such substitution corresponds to a ring homomorphism $\sigma:R\to \mathbb{L}$ such that $\sigma(t_i)=\tau_i$. In this notation, we write $\xi|_{(t_i)=(\tau_i)}=\bar{\sigma}$. Obviously, $\bar{\sigma}$ is a well-defined injective map if and only if for different points ${p},{q}\in \mathcal{P}$, not all coordinates of $\mathbf{u}_{p} \times \mathbf{u}_{q}$ belong to $\ker\sigma$. Moreover, $\bar{\sigma}(\mathcal{P})$ spans the plane if there are points $p_1,p_2,p_3 \in \mathcal{P}$ such that $\det(\mathbf{u}_{p_1},\mathbf{u}_{p_2},\mathbf{u}_{p_3})\not\in \ker\sigma$. 

\begin{definition}
Let $\mathbb{L}/\K$ be a field extension. The ring homomorphism $\sigma:R\to \mathbb{L}$ is \textit{$\xi$-admissible} if $\bar{\sigma}$ is a well-defined injective map which spans $\PG(2,\mathbb{L})$. The prime ideal $P$ of $R$ is \textit{$\xi$-admissible} if the natural homomorphism from $R$ to the quotient field $Q(R/P)$ of $R/P$ is $\xi$-admissible.
\end{definition}

For the map $\xi$ we define the following ideals:
\begin{align*}
A'_\xi &= \cap _{p,q\in\mathcal{P}, p\neq q} \langle \text{coordinates of }\mathbf{u}_{p} \times \mathbf{u}_{q} \rangle, \\
A''_\xi &= \langle \det(\mathbf{u}_{p_1},\mathbf{u}_{p_2},\mathbf{u}_{p_3}) \mid \text{$p_1,p_2,p_3$ different elements of $\mathcal{P}$} \rangle \\
A_\xi &= A'_\xi \cap A''_\xi
\end{align*}
We call $A_\xi$ the \textit{admissibility ideal} of $\xi$. 

\begin{lemma}
The prime ideal $P\triangleleft R$ is $\xi$-admissible if and only if the admissibility ideal $A_\xi$ is not contained in $P$. 
\end{lemma}
\begin{proof}
Since $P$ is prime, $A_\xi\subseteq P$ if and only if $A''_\xi \subseteq P$ or for some $p\neq q\in \mathcal{P}$, all coordinates of $\mathbf{u}_{p} \times \mathbf{u}_{q}$ are in $P$. Hence, for the ring homomorphism $\sigma:R\to Q(R/P)$, either $\bar{\sigma}(\mathcal{P})$ is contained in a line, or $\bar{\sigma}$ is not injective. 
\end{proof}

\begin{definition}
The map $\xi:\mathcal{P}\to \PG(2,\K(t_1,\ldots,t_n))$ is a \textit{pre-embedding of $\Sigma$,} if the following hold:
\begin{enumerate}[(1)]
\item For collinear points $\xi(p_1)$, $\xi(p_2)$, $\xi(p_3)$, the points $p_1,p_2,p_3$ are collinear in $\Sigma$. 
\item For any weak projective $\K$-embedding $\beta$ of $\Sigma$, there is a ring homomorphism $\sigma$ from $R$ to the coordinate field $\mathbb{L}$ of $\beta$ such that $\bar{\sigma}$ is a weak projective $\K$-embedding of $\Sigma$, which is projectively equivalent with $\beta$. 
\end{enumerate}
\end{definition}

It is obvious that $\Sigma$ has pre-embeddings. For example, if each coordinates $u_{p},v_{p},w_{p}$ of $\mathbf{u}_{p}$ are different interminates $t_i$. Moreover, since the projective coordinate frame can be chosen such that the line at infinity contains no point $\beta({p})$, we can set $w_{p}=1$ for all ${p}$, and still have a pre-embedding. For the rest of this section, we fix a pre-embedding $\xi$ of $\Sigma$. 

\subsection{Minimal associated prime ideals of weak embeddings}
We present a computational method to determine weak embeddings of finite incidence structures. Define the ideal
\[ I_{\Sigma,\xi} = \langle \det(\mathbf{u}_{p_1},\mathbf{u}_{p_2},\mathbf{u}_{p_3}) \mid \text{ $p_1,p_2,p_3$ collinear in $\Sigma$} \rangle \]
of $R$. 
\begin{lemma} \label{lm:xiadm}
The following objects are essentially the same:
\begin{enumerate}[(i)]
\item Weak projective $\K$-embeddings $\beta$ of $\Sigma$ with coordinate field $\mathbb{L}$;
\item $\xi$-admissible ring homomorphisms $\sigma:R\to \mathbb{L}$ with $I_{\Sigma,\xi} \leq \ker \sigma$;
\item $\xi$-admissible prime ideals $P$ containing $I_{\Sigma,\xi}$.
\end{enumerate}
\end{lemma}
\begin{proof}
We have already seen the equivalence of (ii) and (iii). The implication (ii)$\Rightarrow$(i) is given by $\beta=\bar{\sigma}$. (i)$\Rightarrow$(ii) follows from the fact that $\xi$ is a pre-embedding. 
\end{proof}
Let 
\[I_{\Sigma,\xi}=Q_1\cap \cdots \cap Q_{m_0}\] 
be the primary decomposition of $I_{\Sigma,\xi}$, with associated prime ideals $P_1=\sqrt{Q_1},\ldots,P_{m_0}=\sqrt{Q_{m_0}}$. Choose the indexing such that $P_1,\ldots,P_m$ are the minimal associated prime ideals of $I_{\Sigma,\xi}$, $m\leq m_0$. Then $\sqrt{I_{\Sigma,\xi}}=P_1\cap \cdots \cap P_m$. Notice that the prime ideals $P_1,\ldots,P_m$ correspond to the irreducible components of the variety $\mathbf{v}(I_{\Sigma,\xi})$ over the algebraic closure of $\K$. Any prime ideal $P$ containing $I_{\Sigma,\xi}$ contains a $P_k$ for some $k\in\{1,\ldots,m\}$. 

\begin{proposition} \label{prop:minprimes}
The following are equivalent:
\begin{enumerate}[(i)]
\item $\Sigma$ has a weak projective $\K$-embedding with coordinate field $\bar{\K}$.
\item $\Sigma$ has a weak projective $\K$-embedding with coordinate field $\mathbb{L}$, where $\mathbb{L}$ is a finite extension of $\K$.
\item $\Sigma$ has a weak projective $\K$-embedding.
\item At least one minimal associated prime ideal of $I_{\Sigma,\xi}$ is $\xi$-admissible.
\end{enumerate}
\end{proposition}
\begin{proof}
(i)$\Rightarrow$(ii): Let $\beta$ be a weak projective $\K$-embedding of $\Sigma$, with coordinate field $\bar{\K}$. Let $X$ be the set of coordinate values of $\beta(p)$ for all $p\in \mathcal{P}$. Then $X$ is a finite subset of $\bar{\K}$, hence the field extension $\mathbb{L}$ of $\K$ which is generated by $X$ is finite. Since the image of $\beta$ is in $\PG(2,\mathbb{L})$, we get a weak projective $\K$-embedding with coordinate field $\mathbb{L}$. 

(ii)$\Rightarrow$(iii) is trivial. (iii)$\Rightarrow$(iv): Let $\beta$ be a weak projective $\K$-embedding of $\Sigma$. By Lemma \ref{lm:xiadm}, we have a ring homomorphism $\sigma:R\to \mathbb{L}$ such that $\beta$ is projectively equivalent with $\bar{\sigma}$. The prime ideal $\ker(\sigma)$ contains a minimal associated prime $P_k$. Since $A_\xi \nleq \ker(\sigma)$, $A_\xi \nleq P_k$ and $P_k$ is $\xi$-admissible. 

(iv)$\Rightarrow$(i): Let $P_k$ be a $\xi$-admissible associated prime of $I_{\Sigma,\xi}$. Since $A_\xi \nleq P_k$, the admissibility ideal $A_\xi$ has a minimal associated prime $P'$ such that $P'\nleq P_k$. This means for the affine varieties over $\bar{\K}$ that $\mathbf{v}(P_k) \nsubseteq \mathbf{v}(P')$. Take an element $(\tau_1,\ldots,\tau_n) \in \mathbf{v}(P_k) \setminus \mathbf{v}(P')$ and define the ring homomorphism $\sigma:R\to \bar{\K}$ by $t_i\mapsto \tau_i$. Then $\bar{\sigma}$ is a weak projective $\K$-embedding of $\Sigma$ with coordinate field $\bar{\K}$. 
\end{proof}

\begin{remark}
\begin{enumerate}[(i)]
\item Let $P$ be a prime ideal of $R$. If a reduced Groebner basis of $P$ is given then the $\xi$-admissibility of $P$ can be decided efficiently using reduction modulo the Groebner basis.
\item There are several mathematical softwares which implement the primary decomposition of ideals of polynomial rings, cf. \cite{Sing_primdec,Macaulay2}. The commands can be accessed using the SageMath \cite{Sage8.2} interface. For the mathematical background and implementation of primary decompositions of polynomial ideal see \cite[Chapter 4]{Singular_book}. 
\item The explicit computation of the admissibility ideal $A_\xi$ is not advisable. It is more efficient to check that neither the ideals generated by the coordinates of $\mathbf{u}_{p} \times \mathbf{u}_{q}$, nor $A''_\xi$ are contained in $P$. 
\end{enumerate}
\end{remark}

\subsection{Example}
We finish this section with an example. Let $\mathcal{P}=\{1,\ldots,9\}$, \[\mathcal{L}= \{\{1,4,7\}, \{1,5,8\}, \{1,6,9\}, \{2,4,8\}, \{2,5,9\}, \{2,6,7\}, \{3,4,9\}, \{3,5,7\}, \{3,6,8\}\}.\] 
Then $\Sigma=(\mathcal{P},\mathcal{L})$ is the dual $3$-net of the cyclic group of order $3$. Define the projective points
\[\begin{array}{lll}
\mathbf{u}_1=[1,0,0], & \mathbf{u}_2=[t_{1},t_{2},t_{13}], & \mathbf{u}_3=[t_{3},t_{4},t_{13}], \\
\mathbf{u}_4=[0,1,0], & \mathbf{u}_5=[t_{5},t_{6},t_{13}], & \mathbf{u}_6=[t_{7},t_{8},t_{13}], \\
\mathbf{u}_7=[1,1,0], & \mathbf{u}_8=[t_{9},t_{10},t_{13}], & \mathbf{u}_9=[t_{11},t_{12},t_{13}]\end{array}\]
with indeterminates $t_1,\ldots,t_{13}$. The map $\xi:i\to \mathbf{u}_i$ is a pre-embedding of $\Sigma$. Indeed, if the last coordinate of any of $\mathbf{u}_2,\mathbf{u}_3, \mathbf{u}_5,\mathbf{u}_6, \mathbf{u}_8,\mathbf{u}_9$ is $0$, then each $\mathbf{u}_i$ is contained in the line $W=0$ by Lemma \ref{lm:lengthdef}. If this is not the case then the last coordinates of $\mathbf{u}_2,\ldots$ can be chosen to be $1$. The ideal $I_{\Sigma,\xi}$ has associated minimal prime ideals
\begin{align*}
P_1&=\langle t_{13} \rangle,\\
P_2&=\langle t_{8}-t_{12}, t_{6}-t_{10}, t_{4}+t_{5}-t_{10}-t_{11}, t_{3}-t_{11}, t_{2}+t_{7}-t_{9}-t_{12}, t_{1}-t_{9}, \\ &t_{5}t_{7}-t_{5}t_{9}-t_{7}t_{11}+t_{9}t_{10}+t_{9}t_{11}-t_{9}t_{12}-t_{10}t_{11}+t_{11}t_{12} \rangle.
\end{align*}
While the ideal $P_2$ is $\xi$-admissible, $P_1$ is not, since all points of the corresponding weak embedding $\mathcal{P}\to \PG(2,Q(R/P_1))$ are contained in a line.

\section{Weak embeddings of multinets of order 6}

In this section, $\mathcal{M}=M_i$, $i\in \{1,\ldots,16\}$, is an abstract light dual multinet of order $6$ with a unique superline of length $2$ or $3$. We assume that $\mathcal{M}$ is well-indexed w.r.t. $\ell$. Let $\K$ be a field and define the polynomial ring
\[R=\K[t_1,\ldots,t_{17}]. \]
\begin{lemma} \label{lm:xi_for_M}
The set
\[\begin{array}{lll}
\mathbf{p}_{1}=[1,0,0] & \mathbf{p}_{7}=[0,1,0] & \mathbf{p}_{13}=[1,1,0] \\
\mathbf{p}_{2}=[1,t_1,0] & \mathbf{p}_{8}=[1,t_2,0] & \mathbf{p}_{14}=[1,t_3,0] \\
\mathbf{p}_{3}=[t_{4},t_{5},1] & \mathbf{p}_{9}=[1,0,1] & \mathbf{p}_{15}=[0,0,1] \\
\mathbf{p}_{4}=[t_{6},t_{7},1] & \mathbf{p}_{10}=[t_{12},t_{13},1] & \mathbf{p}_{16}=[t_{6},t_{13},1] \\
\mathbf{p}_{5}=[t_{8},t_{9},1] & \mathbf{p}_{11}=[t_{14},t_{15},1] & \mathbf{p}_{17}=[t_{8},t_{15},1] \\
\mathbf{p}_{6}=[t_{10},t_{11},1] & \mathbf{p}_{12}=[t_{16},t_{17},1] & \mathbf{p}_{18}=[t_{10},t_{17},1] \\
\end{array}\]
of points gives a pre-embedding of $\mathcal{M}=M_i$, $i\in \{3,\ldots,16\}$.
\end{lemma}
\begin{proof}
The points $\mathbf{p}_{1}, \mathbf{p}_{2}, \mathbf{p}_{7}, \mathbf{p}_{8}, \mathbf{p}_{13}, \mathbf{p}_{14}$ are on the line at infinity $W=0$. If any other point lies on this line, then all do so by Lemma \ref{lm:lengthdef} and the image of $\xi$ is contained in a line. Hence, we may write $1$ for the last coordinate of these points. For the same reason, we can choose the system of projective coordinates such that $\mathbf{p}_{1}, \mathbf{p}_{7}, \mathbf{p}_{9}, \mathbf{p}_{13}, \mathbf{p}_{15}$ have the given form. The coordintes of $\mathbf{p}_{3}, \ldots, \mathbf{p}_{6}, \mathbf{p}_{10}, \ldots, \mathbf{p}_{12}$ are generic. Finally the coordinates of $\mathbf{p}_{16}, \mathbf{p}_{17}$ and $\mathbf{p}_{18}$ have these forms since $\mathcal{M}$ is well-indexed. 
\end{proof}

\begin{table} 
\caption{Minimal associated prime ideals of $I_{\mathcal{M},\xi}$\label{table:primdec}}
\begin{tabular}{|c|c|c|c|}
\hline
$\mathcal{M}$& nr of $P_k$'s & nr of $\mathcal{M}$-admissible $P_k$'s & $\dim$ of $\mathcal{M}$-admissible $P_k$\\ \hline
$M_{3}$ & 6 & 1 & 2 \\
$M_{4}$ & 3 & 1 & 2 \\
$M_{5}$ & 5 & 0 & - \\ 
$M_{6}$ & 1 & 0 & - \\
$M_{7}$ & 5 & 0 & - \\
$M_{8}$ & 1 & 1 & 1 \\
$M_{9}$ & 2 & 1 & 1 \\
$M_{10}$ & 4 & 1 & 1 \\
$M_{11}$ & 2 & 0 & - \\
$M_{12}$ & 3 & 1 & 2 \\
$M_{13}$ & 2 & 1 & 1 \\
$M_{14}$ & 4 & 1 & 1 \\
$M_{15}$ & 1 & 1 & 1 \\
$M_{16}$ & 1 & 1 & 1 \\
\hline 
\end{tabular}
\end{table}

For $\mathcal{M}=M_i$, $i\in \{3,\ldots,16\}$ and for $\K=\mathbb{Q}$, the associated minimal prime ideals of $I_{\mathcal{M},\xi}$ can be explicitly computed using \cite{GAP_Sing,Singular,Sing_primdec,Macaulay2} and the SageMath \cite{Sage8.2} interface. We summarize the computational results in the following lemma.

\begin{lemma} \label{lm:table2}
Let $\xi$ be the pre-embedding of Lemma \ref{lm:xi_for_M}. The number of associated prime ideals and the number of $\xi$-admissible associated prime ideals of $I_{M_i,\xi}$, $i\in \{3,\ldots,16\}$, are given in Table \ref{table:primdec}. In particular, $I_{M_i,\xi}$ has at most one $\xi$-admissible associated prime ideal. The Krull dimension of the $\xi$-admissible associated prime ideal is given in the last column of the Table \ref{table:primdec}.
\end{lemma}

We state our main result now.

\begin{theorem} \label{thm:main}
\begin{enumerate}[(i)]
\item Any weak projective $\mathbb{K}$-embedding of $M_1$ and $M_2$ is either conic-line or tetrahedron type. 
\item The abstract light dual multinets $M_3$, $M_4$, $M_8$, $M_9$, $M_{10}$, $M_{12}$, $M_{13}$, $M_{14}$, $M_{15}$, $M_{16}$ have a weak projective embedding in $\PG(2,\mathbb{C})$. 
\item The abstract light dual multinets $M_5$, $M_6$, $M_7$, $M_{11}$ have no weak projective embeddings in $\PG(2,\mathbb{C})$. 
\end{enumerate}
\end{theorem}
\begin{proof}
(i) follows from Proposition \ref{pr:superlinr3} and \cite[Proposition 4.3]{KN_multinets}. 
Proposition \ref{prop:minprimes} and Lemma \ref{lm:table2} imply (ii) and (iii). 
\end{proof}

\section{Merged blocks of weak embeddings}

In this section, we compute the generic merged blocks of the weak projective $\mathbb{Q}$-embeddings of the abstract light dual multinets $M_3$, $M_4$, $M_8$, $M_9$, $M_{10}$, $M_{12}$, $M_{13}$, $M_{14}$, $M_{15}$, $M_{16}$. 

If $\beta$ is given by a pre-embedding $\xi$ in $\PG(2,\mathbb{Q}(t_1,\ldots,t_n))$ and a prime ideal $P$ of $R=\mathbb{Q}[t_1,\ldots,t_n]$, containing $I_{\Sigma,\xi}$, then we can speak of the merged blocks of $P$. Indeed in this case, the coordinate field $\mathbb{L}$ is the quotient field of $R/P$. The merged blocks can be determined by the triples $p_1,p_2,p_3\in \mathcal{P}$ for which 
\[\det(\mathbf{u}_{p_1},\mathbf{u}_{p_2},\mathbf{u}_{p_3}) \in P.\]
This enables us to compute the merged blocks of the weak embeddings of light dual multinets of order $6$ with a unique superline of length $2$. 

\begin{lemma}
If $\Sigma$ is an abstract light dual multinet, $\beta:\mathcal{P}\to \PG(2,\mathbb{L})$ is a weak projective embedding of $\Sigma$, then any merged block $B\subseteq \mathcal{P}$ is either contained in one of the components $\mathcal{P}_1$, $\mathcal{P}_2$ or $\mathcal{P}_3$, or, it is a new line of length at least $2$ in the sense of Lemma \ref{lm:lengthdef}. 
\end{lemma}
\begin{proof}
Denote by $S_i$ the subsets of $Q$, whose elements correspond to the points in $B\cap \mathcal{P}_i$, $i=1,2,3$. If two of $S_1,S_2,S_3$ are not empty then $S_1\cdot S_2\subseteq S_3$, $S_1\setminus S_3\subseteq S_2$ and $S_3/S_2\subseteq S_1$ hold. This shows that $(S_1,S_2,S_3)$ is a subsquare and  $|S_1|=|S_2|=|S_3|$. 
\end{proof}

\begin{table}[]
\caption{Merged blocks for light multinets of order $6$\label{table:merged}}
\begin{tabular}{|c|c|c|c|c|}
\hline
$\mathcal{M}$ & nr of new long lines & sizes in $\mathcal{P}_1$ & sizes in $\mathcal{P}_2$ & sizes in $\mathcal{P}_3$   \\ \hline
$M_{3}$  & 2        &       &       &             \\ 
$M_{4}$  &          & 3, 3   & 3, 3   & 3, 3      \\ 
$M_{8}$  &          &       &       & 5           \\ 
$M_{9}$  &          &       &       & 3, 3        \\ 
$M_{10}$ &          &       &       & 3, 3, 3, 3  \\ 
$M_{12}$ &          &       &       & 3           \\ 
$M_{13}$ &          &       &       &             \\ 
$M_{14}$ & 1        & 3, 3   & 3, 3   & 3, 3  1   \\ 
$M_{15}$ &          & 3, 3   &       &            \\ 
$M_{16}$ & 1        &       &       & 3, 3        \\ \hline
\end{tabular}
\end{table}

\begin{lemma}
Let $\mathcal{M}$ be one of $M_3$, $M_4$, $M_8$, $M_9$, $M_{10}$, $M_{12}$, $M_{13}$, $M_{14}$, $M_{15}$, $M_{16}$ and let $\xi$ be the pre-embedding defined above. Then there is a unique $\xi$-admissible minimal associated prime ideal $P$ of $I_{\mathcal{M},\xi}$. The merged blocks of $P$ are given in Table \ref{table:merged}. 
\end{lemma}
\begin{proof}
The uniqeness of $P$ follows from Table \ref{table:primdec}. For given $P$, the data of Table \ref{table:merged} can be computed efficienty with the Singular \cite{Singular} package \cite{Sing_primdec}. 
\end{proof}

\begin{proposition}
Let $\mathcal{M}$ be an abstract light dual multinet of order $6$ with a unique superline of length $2$. Assume that $\mathcal{M}$ has a weak projective $$\mathbb{Q}$$-embedding $\beta$. Then $\beta$ has at least the merged blocks given in Table \ref{table:merged}. 
\end{proposition}
\begin{proof}
Let $P$ be the unique minimal associated prime ideal of $I_{\mathcal{M},\xi}$. Let $\beta:\mathcal{P}\to \PG(2,\mathbb{L})$ be a weak projective $\mathbb{Q}$-embedding of $\mathcal{M}$. By Lemmas \ref{lm:xiadm} and \ref{lm:xi_for_M}, there is a ring homomorphism $\sigma:R\to \mathbb{L}$ such that $\beta=\bar{\sigma}$. The uniqueness of $P$ implies $P\leq \ker(\sigma)$. Therefore, the merged blocks of $P$ are merged blocks of $\beta$. 
\end{proof}

\appendix

\section{SageMath code for well-indexing of finite abstract light dual multinets}

Let $\Sigma=(\mathcal{P},\mathcal{B})$ be an abstract light dual multinet. This function picks a block $B\in \mathcal{B}$ of maximal length and returns an isomorphic abstract light dual multinet $\Sigma'$ which is well-indexed w.r.t. $B$. 

\begin{lstlisting}[language=python,basicstyle=\footnotesize]
def blockdesign_wellindexing(des):
    n=int(len(des.ground_set())/3)
    r=max(des.block_sizes())/3
    dn=des.blocks()
    sl=[d for d in dn if len(d)==3*r].pop()
    a=[0]*3*n
    a[0:r]=tuple(sl[0:r])
    a[n:n+r]=tuple(sl[r:2*r])
    a[2*n:2*n+r]=tuple(sl[2*r:3*r])
    a[n+r:2*n]=[x for x in range(n,2*n) if not(x in sl)]
    a[2*n+r:]=[[d[2] for d in dn if (d[0],d[1])==(a[0],a[n+j])].pop()
        for j in range(r,n)]
    a[r:n]=[[d[0] for d in dn if (d[1],d[2])==(a[n],a[2*n+j])].pop() 
        for j in range(r,n)]
    aa=[a.index(i) for i in range(3*n)]
    dnn=[tuple([aa[x] for x in d]) for d in dn]
    dnn.sort()
    return BlockDesign(3*n,dnn)
\end{lstlisting}

\section{SageMath code for the classification of abstract light dual multinets of order $6$}

\begin{lstlisting}
cts=[
    ([[1,2,3,4,5,6],[2,3,4,5,6,1],[3,4,5,6,1,2],
        [4,5,6,1,2,3],[5,6,1,2,3,4],[6,1,2,3,4,5]],"#6.1.1.1"),
    ([[1,2,3,4,5,6],[2,1,5,6,3,4],[3,6,1,5,4,2],
        [4,5,6,1,2,3],[5,4,2,3,6,1],[6,3,4,2,1,5]],"#6.2.1.1"),
    ([[1,2,3,4,5,6],[2,3,1,5,6,4],[3,1,2,6,4,5],
        [4,6,5,2,1,3],[5,4,6,3,2,1],[6,5,4,1,3,2]],"#6.3.1.1"),
    ([[1,2,3,4,5,6],[2,1,4,3,6,5],[3,4,5,6,1,2],
        [4,3,6,5,2,1],[5,6,1,2,4,3],[6,5,2,1,3,4]],"#6.4.1.1"),
    ([[1,2,3,4,5,6],[2,1,5,6,3,4],[3,6,2,5,4,1],
        [4,5,6,2,1,3],[5,4,1,3,6,2],[6,3,4,1,2,5]],"#6.5.1.1"),
    ([[1,2,3,4,5,6],[2,1,4,5,6,3],[3,6,2,1,4,5],
        [4,5,6,2,3,1],[5,3,1,6,2,4],[6,4,5,3,1,2]],"#6.6.1.1"),
    ([[1,2,3,4,5,6],[2,1,4,3,6,5],[3,5,1,6,4,2],
        [4,6,5,1,2,3],[5,3,6,2,1,4],[6,4,2,5,3,1]],"#6.7.1.1"),
    ([[1,2,3,4,5,6],[2,1,6,5,3,4],[3,6,1,2,4,5],
        [4,5,2,1,6,3],[5,3,4,6,1,2],[6,4,5,3,2,1]],"#6.8.1.1"),
    ([[1,2,3,4,5,6],[2,3,1,6,4,5],[3,1,2,5,6,4],
        [4,6,5,1,2,3],[5,4,6,2,3,1],[6,5,4,3,1,2]],"#6.9.1.1"),
    ([[1,2,3,4,5,6],[2,1,6,5,4,3],[3,5,1,2,6,4],
        [4,6,2,1,3,5],[5,3,4,6,2,1],[6,4,5,3,1,2]],"#6.10.1.1"),
    ([[1,2,3,4,5,6],[2,1,4,5,6,3],[3,4,2,6,1,5],
        [4,5,6,2,3,1],[5,6,1,3,2,4],[6,3,5,1,4,2]],"#6.11.1.1"),
    ([[1,2,3,4,5,6],[2,1,5,6,4,3],[3,5,4,2,6,1],
        [4,6,2,3,1,5],[5,4,6,1,3,2],[6,3,1,5,2,4]],"#6.12.1.1")
]

def table_to_dualnet(t):
    n=len(t)
    ret=[]
    for y in range(n): 
        ret.extend([(x,y+n,t[x][y]-1+2*n) for x in range(n)])
    return ret

def generated_subsquare(dn,gens):
    while True:
        bls=[]
        for a in Combinations(gens,2):
            a=set(a)
            bls.extend([d for d in dn if len(a.intersection(d))>1])
        pts=set()
        for a in bls:
            pts=pts.union(set(a))
        if len(pts)>len(gens):
            gens=pts
        else:
            gens=list(gens)
            gens.sort()
            return tuple(gens),bls

def all_proper_subsquares(dn):
    n=sqrt(len(dn))
    ret=[]
    for x in Combinations(range(n),2):
        for y in range(n,2*n):
            new=generated_subsquare(dn,{x[0],x[1],y})[0]
            if (len(new)<3*n) and not(new in ret): ret.append(new)
    ret.sort()
    return ret

def superline_to_blocks(dn,sl):
    n=sqrt(len(dn))
    ret=[d for d in dn if not((d[0] in sl) and (d[1] in sl))]
    ret.append(tuple(sl))
    return ret #BlockDesign(3*n,ret)

##############

all_abstract_ldms=[]
for ct in cts:
    dn=table_to_dualnet(ct[0])
    sqs=all_proper_subsquares(dn)
    for sl in sqs:
        bls=superline_to_blocks(dn,sl)
        des=BlockDesign(18,bls)
        des.qgname=ct[1]
        all_abstract_ldms.append(des)

len(all_abstract_ldms)

classes=[]
inds=range(len(all_abstract_ldms))
while len(inds)>0:
    new_class=[i for i in inds if 
        all_abstract_ldms[inds[0]].is_isomorphic(all_abstract_ldms[i])]
    classes.append(new_class)
    inds=[i for i in inds if not(i in new_class)]

len(classes)

ldms_ids=[
    ("#6.9.1.1", "#6.1.1.1", "(((C3 x C3 x C3) : C3) : C2) : C2"),
    ("#6.9.1.1", "#6.3.1.1", "#6.2.1.1", 
        "(((C3 x C3 x C3) : C3) : C2) : C2"),
    ("#6.4.1.1", "#6.1.1.1", "C2 x S4"),
    ("#6.5.1.1", "#6.2.1.1", "C2 x S4"),
    ("#6.4.1.1", "C2 x S4"),
    ("#6.7.1.1", "C2 x S4"),
    ("#6.5.1.1", "C2 x D4"),
    ("#6.11.1.1", "#6.8.1.1", "C2 x D4"),
    ("#6.10.1.1", "C2 x D4"),
    ("#6.11.1.1", "C2 x D4"),
    ("#6.9.1.1", "#6.12.1.1", "C2 x C2 x C2"),
    ("#6.12.1.1", "C2 x C2 x C2"),
    ("#6.7.1.1", "#6.6.1.1", "S3"),
    ("#6.5.1.1", "#6.10.1.1", "C2 x C2"),
    ("#6.7.1.1", "#6.10.1.1", "C2 x C2"),
    ("#6.11.1.1", "#6.12.1.1", "C2 x C2"),
]
abstract_ldms=dict()
for cl in classes:
    names=set(all_abstract_ldms[i].qgname for i in cl)
    des=blockdesign_wellindexing(all_abstract_ldms[cl[0]])
    des.qgrs=names
    #abstract_ldms.append(des)
    stdescr=des.automorphism_group().structure_description()
    pos=list(names)
    pos.append(stdescr)
    print pos
    pos=ldms_ids.index(tuple(pos))
    abstract_ldms[pos]=des
\end{lstlisting}

\section{SageMath code for the weak embeddings of abstract light dual multinets of order $6$}

\begin{lstlisting}
P=PolynomialRing(QQ,'t',17,order='lex')
t=P.gens()

def ideal_of_preembedding(xi,des):
    eqs=[]
    for bl in des.blocks():
        for t in Combinations(bl,3):
            eqs.append(det(Matrix([xi[i] for i in t])))
    eqs=[x for x in eqs if not x.is_zero()]
    I=P.ideal(eqs)
    return I

def pairwise_crossproducts(pts):
    ret=[]
    for a in Combinations(pts,2):
        u=vector(a[0]),vector(a[1])
        ret.append(u[0].cross_product(u[1]))
    return ret

def not_all_zero_modI(pts,I):
    for p in pts:
        if ([I.reduce(x*P.one()).is_zero() for x in p]==[True]*3):
            return False
    return True

def merged_blocks_of_embedding(xi,pi):
    bls=[]
    for bl in Combinations(range(len(xi)),3):
        d=det(Matrix([xi[i] for i in bl]))
        d=pi.reduce(d*P.one())
        if d.is_zero():
            bls.append(set(bl))
    bl=0
    while bl<len(bls):
        d=[i for i in range(bl+1,len(bls)) 
            if len(bls[bl].intersection(bls[i]))>1]
        if d==[]:
            bl=bl+1
        else:
            d.reverse()
            for i in d:
                bls[bl]=bls[bl].union(bls[i])
                bls.remove(bls[i])
    bls=[tuple(sorted(bl)) for bl in bls] 
    return bls

xi=[
    [    1,    0,0],[    1, t[0],0],[ t[3], t[4],1],
    [ t[5], t[6],1],[ t[7], t[8],1],[ t[9],t[10],1],
    [    0,    1,0],[    1, t[1],0],[    1,    0,1],
    [t[11],t[12],1],[t[13],t[14],1],[t[15],t[16],1],
    [    1,    1,0],[    1, t[2],0],[    0,    0,1],
    [ t[5],t[12],1],[ t[7],t[14],1],[ t[9],t[16],1]
]
xi_crossprods=pairwise_crossproducts(xi)

##############

pds=[]
for i in [2..15]:
    I=ideal_of_preembedding(xi,abstract_ldms[i])
    minprimes=I.minimal_associated_primes()
    dd=[]
    for pi in minprimes:
        dd.append(not_all_zero_modI(xi_crossprods,pi))
        if dd[-1]: pds.append([i,pi])
    print "# ",i,"\t",dd

[x[1].dimension() for x in pds]

for a in pds:
    a.append(merged_blocks_of_embedding(xi,a[1]))

for a in pds:
    print "# ",a[0],"\tdim=",a[1].dimension(),"\t",
    print [len([b for b in a[2] if ((6*i<=b[0]) and (b[-1]<6*(i+1)))]) 
        for i in range(3)],"\t",
    bb=[len(x) for x in a[2]]
    print [[i,bb.count(i)] for i in set(bb)]
\end{lstlisting}

\end{document}